\documentclass[a4paper]{amsart}
\usepackage{graphicx} 

\usepackage{enumitem}
\usepackage{bbm}

\usepackage[hidelinks]{hyperref}

\usepackage{amsmath}
\usepackage{mathtools}
\usepackage{amssymb}
\usepackage{amsfonts}
 
\newcommand{\natNum}{\mathbb{N}}
\newcommand{\realNum}{\mathbb{R}}
\newcommand{\complNum}{\mathbb{C}}
\newcommand{\posRealNum}{[0,\infty)}
\newcommand{\semiGroup}[1]{\mathbb{T}(#1)}
\newcommand{\semiGroupDef}{\mathbb{T} = \left( \semiGroup{t} \right)_{t \geq 0}}
\newcommand{\semiGroupEx}[1]{\mathbb{T}_{-1}(#1)}
\newcommand{\loc}{\mathrm{loc}}
\newcommand{\andMath}{\textrm{and}}
\newcommand{\realPart}[1]{\mathrm{Re}\left( #1 \right)}
\newcommand{\set}[1]{\left\lbrace #1 \right\rbrace}
\newcommand{\innerProd}[2]{\left\langle #1 , #2 \right\rangle}
\newcommand{\boundedOp}[2]{\mathcal{L} \left( #1 , #2 \right)}
\newcommand{\boundedOpSelf}[1]{\mathcal{L} \left( #1  \right)}
\newcommand{\lpSpaceDT}[3]{L^{#1}\left(#2, #3\right)}
\newcommand{\lpSpacelocDT}[3]{L^{#1}_{\loc}\left(#2, #3\right)}
\newcommand{\CinfComp}[2]{C^\infty_c\left(#1, #2\right)}
\newcommand{\laplaceTr}[1]{\mathcal{L}\left\lbrace #1 \right\rbrace}
\newcommand{\intd}[1]{\,\mathrm{d}#1}
\newcommand{\derOrd}[3]{\frac{\mathrm{d}^{#3}#1}{\mathrm{d}#2^{#3}}}
\newcommand{\inv}[1]{{#1}^{-1}}
\newcommand{\systemNode}[4]{\Sigma(#1,#2,#3,#4)}
\newcommand{\systemNodeSt}{\systemNode{A}{B}{C}{\mathbf{G}}}
\newcommand{\systemNodeStD}{\systemNode{A^*}{C^*}{B^*}{\overline{\mathbf{G}}(\overline{\mkern2mu\cdot\mkern2mu})}}
\newcommand{\CandD}{C \& D}
\newcommand{\resolvent}[2]{\left( #1 \mathbb{I} - #2 \right)^{-1}}
\newcommand{\dom}{\mathrm{D}}
\newcommand{\idMatrix}[1]{\mathbbm{1}_{#1}}

\usepackage{tikz}
\usetikzlibrary{arrows.meta,bending}
\usepackage{tkz-euclide}
\usetikzlibrary{positioning}

\newtheorem{proposition}{Proposition}[section]
\newtheorem{lemma}[proposition]{Lemma}
\newtheorem{theorem}[proposition]{Theorem}
\newtheorem{corollary}[proposition]{Corollary}

\theoremstyle{definition}
\newtheorem{definition}[proposition]{Definition}
\newtheorem{example}[proposition]{Example}

\theoremstyle{remark}
\newtheorem{remark}[proposition]{Remark}

\allowdisplaybreaks
\mathtoolsset{showonlyrefs}

\begin{document}

    \title{A dual notion to BIBO stability}
    \date{Submitted to the editors on August 11, 2024.}

\author{Felix L. Schwenninger}
\address{Department~of~Applied~Mathematics,~University~of~Twente, P.O.~Box~217, 7500~AE Enschede, The~Netherlands}
\email{f.l.schwenninger@utwente.nl}

\author[Alexander A. Wierzba]{Alexander A. Wierzba*}
\address{Department~of~Applied~Mathematics,~University~of~Twente, P.O.~Box~217, 7500~AE Enschede, The~Netherlands}
\email{a.a.wierzba@utwente.nl}

\thanks{*Corresponding author}
\thanks{Wierzba is supported by the Theme Team project ``Predictive Avatar Control and Feedback'' sponsored by the Faculty of Eletrical Engineering, Computer Science and Mathematics.}
    
\begin{abstract}
    In this paper we consider BIBO stability of infinite-dimensional linear state-space systems and the related notion of $L^1$-to-$L^1$ input-output stability (abbreviated \emph{LILO}).
    We show that in the case of finite-dimensional input and output spaces, both are equivalent and preserved under duality transformations.
    In the general case, neither of these properties is satisfied, but BIBO and LILO stability remain dual to each other.
\end{abstract}


\maketitle

\section{Introduction}
While {bounded-input-bounded-output stability (BIBO)} is a classical notion in the control of finite-dimensional systems \cite{b_Kailath80_LinearSystems, b_Vidyasagar93_NonlinearSystemsAnalysis, b_PoldermanWillems98_SystemsTheory} its infinite-dimensional counterpart has been explored much less. This may stem from the fact that infinite-dimensional systems theory is most commonly studied in the context of operator theory on Hilbert spaces \cite{b_CurtainZwart2020,b_TucsnakWeiss2009} --- rather than on general Banach spaces ---, for which input-output stability can for instance be cast in $L^{2}$-wellposedness, \cite{b_Staffans2005}. Yet, estimates for the input and/or output functions in supremum type norms are relevant; as they naturally appear in e.g.\ funnel control \cite{a_BergerIlchmannRyanFunnel2021} and input-to-state stability \cite{b_KarafyllisKristic18_ISSforPDEs,a_MironchenkoPrieur20ISSSurvey} also in particular in the setting of infinite-dimensional state space systems \cite{a_BergerPucheSchwenningerFunnel2020}.

This note complements the recent systematic study of BIBO stability for infinite-dimensional linear systems in \cite{a_SchwenningerWierzbaZwart24BIBO}  by addressing the (formally dual) situation where input and output are measured in $L^{1}$-norms.

Our interest in this setting is not artificial as it arises from two goals:
First, to find sufficient conditions for BIBO stability and second to understand what happens to BIBO stability when considering dual systems. 

The challenge of assessing BIBO stability for linear infinite-dimensional systems is already present in the case where the input and output spaces are finite-dimensional, see e.g.\ \cite{a_BergerPucheSchwenningerFunnel2020}. The ad-hoc technique used there relies on showing that a given transfer transfer function is the Laplace transform of a measure of bounded variation. The current exclusion of infinite-dimensional input and output spaces is restrictive in cases and thus a full investigation and characterisation of the general setting is desirable. In light of this, both the study of the relationship between BIBO stability and the corresponding $L^1$-to-$L^1$ stability notion as well as of the behaviour of BIBO stability under duality transformations are also part of the larger program tackling this general question.

In the remainder of the introduction we take the liberty to showcase the main questions addressed in this paper by relating to the (trivial) finite-dimensional situation.

For a  finite-dimensional LTI system in state space form, i.e. $\Sigma(A,B,C,D)$,
\begin{equation}
\label{eq:formalStateSpaceEquations}
\begin{array}{ccc}
\dot{x}(t)&=&Ax(t)+Bu(t)\\
y(t)&=&Cx(t)+Du(t)\end{array}\quad t\ge0,
\end{equation}
with matrices $A\in\complNum^{n_X \times n_X}$, $B\in\complNum^{n_X \times n_U}$, $C\in\complNum^{n_Y \times n_X}$ and $D\in\complNum^{n_Y \times n_U}$, BIBO stability is equivalent to  absolute integrability of the (matrix-valued) function  $t \mapsto C e^{At} B$.

If one instead considers distributed parameter systems such as arising in the modelling of wave or flow phenomena using partial differential equations, one encounters systems which are still formally of the form~\eqref{eq:formalStateSpaceEquations} but allow $x$, $u$ and $y$ to take values in some infinite-dimensional Banach spaces. This generalisation results in a more subtle situation in which already defining BIBO stability becomes more difficult \cite{a_SchwenningerWierzbaZwart24BIBO}. In particular there are conclusions from the finite-dimensional case that do not carry over to the infinite-dimensional situation. For example, exponential stability implies BIBO stability in finite dimensions, but not in infinite-dimensional state spaces in general. Note that the expression $C e^{At} B$ will not be well-defined in the context of controls and observations acting on spatial boundaries as the operators $C$ and $B$ fail to be bounded.

Another observation in the finite-dimensional case is that the notion of BIBO stability is preserved under \emph{duality transforms}, that is the \emph{dual system} $\Sigma(A^{*},C^{*},B^{*},D^{*})$
is BIBO stable if and only if $\Sigma(A,B,C,D)$ is. 
We show that the corresponding statement in infinite-dimensions is true if the input and the output spaces are finite-dimensional, see Proposition~\ref{prop:finiteDimSpacesBIBODuality}, and provide a counterexample in the general case.

Studying the relation of BIBO stability and duality transforms furthermore directly leads to questions regarding the connection between BIBO stability and the analogous concept of $L^1$-to-$L^1$ input-output stability (or short \emph{LILO} stability). In fact, it will turn out that in the case of finite-dimensional input and output spaces these two stability notions are equivalent, see Corollary~\ref{cor:BIBOLILOequivalence}.

While we will provide a counterexample showing that also this equivalence fails in the general case, we will see that LILO stability is in fact the appropriate dual notion to BIBO stability in the view of a system and its dual system, see Theorems~\ref{thm:CinfLILOImplDualLInfBIBO}~and~\ref{thm:LinfBIBODualCInfLILO}. 

\subsection{Notation}
For any $\alpha \in \realNum$, let $\complNum_\alpha := \left\lbrace z \in \complNum \;\middle|\; \realPart{z} > \alpha \right\rbrace$. All Banach spaces are considered over the field of complex numbers. Let $\boundedOp{X}{Y}$ be the set of bounded operators from a Banach space $X$ to a Banach space $Y$ and denote by $\dom(A)$ the domain of an unbounded operator $A$.

For a measure $h$, let $\laplaceTr{h}$ denote its Laplace transform if it exists (on some right half-plane).

Let $\mathcal{M}(\posRealNum, \mathbb{K}^{n \times m})$ denote the set of Borel measures of bounded total variation on $\posRealNum$ with values in $\mathbb{K}^{n \times m}$. 
Furthermore, for $h \in \mathcal{M}(\posRealNum, \mathbb{K}^{n \times m})$ let $\| h \|_{\mathcal{M}}$ denote the total variation of $h$ \cite[Sec.~3.2]{b_GripenbergLondenStaffans1990}.

For any $1 \leq p \leq \infty$, let $1 \leq p' \leq \infty$ be the real number satisfying $\frac{1}{p} + \frac{1}{p'} = 1$.

\section{System nodes and BIBO stability}
We begin by recalling the notion of system nodes in the sense of Staffans \cite{b_Staffans2005}\cite[Def.~4.1]{a_TucsnakWeissWellPosed2014} which provide  a formally rigorous extension of systems of the form \eqref{eq:formalStateSpaceEquations} to the case of infinite-dimensional state spaces, followed by the definition of BIBO stability for these types of systems.

Let $U$, $X$ and $Y$ be Banach spaces. Then a \emph{system node} on $(U,X,Y)$ is a collection $\Sigma(A,B,C,\mathbf{G})$ where $A: \dom(A) \subset X \rightarrow X$ is the generator of a strongly continuous semigroup $\semiGroupDef$ on $X$ with growth bound $\omega(\mathbb{T})$, $B \in \boundedOp{U}{X_{-1}}$, $C \in \boundedOp{X_1}{Y}$ and $\mathbf{G}: \complNum_{\omega(\mathbb{T})} \rightarrow \boundedOp{U}{Y}$ is an analytic function satisfying for all $\alpha, \beta \in \complNum_{\omega(\mathbb{T})}$
    \begin{equation}
    \label{eq:transferFunctionDifferenceRelation}
        \mathbf{G}(\alpha) - \mathbf{G}(\beta) = C\left[ \resolvent{\alpha}{A} - \resolvent{\beta}{A} \right] B,
    \end{equation}
where $X_1$ is defined as the space $\dom(A)$ with the norm $\|x\|_{X_1} := \| (\beta \mathbb{I} - A) x \|_X$ with $\beta \in \rho(A)$ and $X_{-1}$ as the completion of $X$ with respect to the norm $\|x\|_{X_{-1}} := \| \resolvent{\beta}{A} x \|_X$ again with $\beta \in \rho(A)$.

 The spaces $U$, $X$ and $Y$ are usually referred to as \emph{input space}, \emph{state space} and \emph{output space}, respectively. Furthermore, $B$ is called the \emph{input operator}, $C$ the \emph{output operator} and $\mathbf{G}$ the \emph{transfer function} of the system node. 

The well-defined form of Equations \eqref{eq:formalStateSpaceEquations} is then given by
\begin{equation}
\label{eq:systemNodeDifferentialEqns}
    \begin{split}
        \Dot{x}(t) &= A x(t) + B u(t) \\
        y(t) &= C\&D \begin{bmatrix} x(t) \\ u(t) \end{bmatrix},
    \end{split}
\end{equation}
where $C \& D: \dom(C \& D) \subset X \times U \rightarrow Y$ is the \emph{combined observation/feedthrough operator} defined as
\begin{equation}
\label{eq:defCandDOperator}
    C\&D \begin{bmatrix} x \\ u \end{bmatrix} := C \left[ x - \resolvent{\beta}{A} B u \right] + \mathbf{G}(\beta) u,
\end{equation}
for some $\beta \in \complNum_{\omega(\mathbb{T})}$ and with domain $\dom(C\&D) = \set{ \left[ \begin{smallmatrix} x \\ u \end{smallmatrix} \right] \in X \times U \middle| Ax + Bu \in X }$. 

Let $\mathcal{F}$ be a placeholder for ``$L^p$'' and ''$C$'', meaning that e.g.\ $\mathcal{F}([0,t],X)$ refers to the space of $X$-valued $L^{p}$ or continuous functions on the interval $[0,t]$, respectively. If there exists a constant $c > 0$ such that for some $t > 0$ and any $x \in X_1$ we have that $\left\| \mathbbm{1}_{[0,t]} C \semiGroup{\cdot} x \right\|_{\mathcal{F}([0,t],Y)} \leq c \| x \|_X$ then $C$ is called \emph{$\mathcal{F}$-observation-admissible}. On the other hand, an operator $B$ is called \emph{$\mathcal{F}$-control-admissible} if for some (hence all) $t > 0$ and any $u:[0,t]\to U \in \mathcal{F}([0,t],Y)$ we have that $\int_0^t \semiGroup{t - s} B u(s) \intd s \in X$, where identify the semigroup with its extension to $X_{-1}$ and we emphasize that the integral is defined in $X_{-1}$. The latter property readily implies that there exists $c > 0$ such that $\| \int_0^t \semiGroup{t - s} B u(s) \intd s \|_X \leq c \| u \|_{\mathcal{F}([0,t],Y)}$ for all $u$. In case that either inequality in the definition of admissible observation and control operators holds with the same constant $c$ for all $t > 0$, the respective operator is called \emph{infinite-time $\mathcal{F}$-admissible}.

If $\Sigma(A,B,C,\mathbf{G})$ is a system node and the Banach space adjoint $A^*$ is again the generator of a strongly continuous semigroup $\mathbb{T}^*$, then we call
$\systemNodeStD$ the \emph{dual} system node to $\Sigma(A,B,C,\mathbf{G})$. This provides the infinite-dimensional generalisation of the dual system in finite dimensions. 

Given Equation \eqref{eq:systemNodeDifferentialEqns} we can then consider two different types of solutions to a system node $\Sigma(A,B,C,\mathbf{G})$. 

First, there are \emph{classical solutions} \cite[Def.~4.2]{a_TucsnakWeissWellPosed2014}, that is triples $(u,x,y)$ of functions $u \in C([0,\infty),U)$, $x \in C^1([0,\infty),X)$ and $y \in C([0,\infty),Y)$ such that $\left[ \begin{smallmatrix}
    x(t) \\u(t)
\end{smallmatrix} \right] \in \dom(C\&D)$ for all $t \geq 0$ and Equations~\eqref{eq:systemNodeDifferentialEqns} holds for all $t \geq 0$.

Second, we have  \emph{generalised solutions in the distributional sense}, that is triples $(u,x,y)$, where $u \in L^1_{\loc}([0,\infty), U)$ and $x \in C([0,\infty), X_{-1})$ are functions with the latter given in terms of $\mathbb{T}_{-1}$, the unique extension of $\mathbb{T}$ to the space $X_{-1}$, by $x(t) = \semiGroup{t} x_0 + \int_0^t \mathbb{T}_{-1}\left( t - s \right) B u(s) \intd{s}$ for some $x_0 \in X$ and all $t \geq 0$, and $y$ is the $Y$-valued distribution given distributionally as 
    \begin{equation}
    \label{eq:distributionalOutputDef}
    y(t) = \derOrd{}{t}{2} \left( \left( C\&D \right) \int_0^t (t - s) \begin{bmatrix} x(s) \\ u(s) \end{bmatrix} \intd{s} \right), \quad t \geq 0. 
    \end{equation}

The existence of classical solutions is guaranteed at least for all $u \in W^{2,1}_{\loc}(\posRealNum, U)$ and initial value $x(0) \in X$ such that $\left[ \begin{smallmatrix}
    x(0) \\ u(0)
\end{smallmatrix} \right] \in \dom(C \& D)$ by \cite[Lem.~4.7.8]{b_Staffans2005}, whereas generalized solutions exist for any $u \in \lpSpacelocDT{1}{\posRealNum}{U}$ and initial value $x_0 \in X$ by \cite[Thm.~3.8.2(i)~\&~Lem.~4.7.9]{b_Staffans2005}. Furthermore, clearly any classical solution is also a generalised solution of the system node.

Following \cite[Def.~3.1~\&~Def.~3.2]{a_SchwenningerWierzbaZwart24BIBO}, we define three variants of BIBO stability depending on the solution concept employed.
\begin{definition}
\label{def:BIBOStability}
A system node $\Sigma(A,B,C,\mathbf{G})$ is called $C^ \infty$- ($C$-, $L^\infty$-)\emph{BIBO stable} if there exists $c > 0$ such that for any classical (generalised distributional) solution $(u,x,y)$ with $u \in C^\infty_c\left( (0,\infty), U \right)$ ($u \in C\left( [0,\infty), U \right)$, $u \in L^\infty_\loc(\posRealNum, U)$) and $x(0) = 0$ we have that $y \in L^\infty_\loc(\posRealNum, Y)$ and for all $t > 0$
\begin{equation}
\label{eq:BIBOInequality}
    \| y \|_{\lpSpaceDT{\infty}{[0,t]}{Y}} \leq c \| u \|_{\lpSpaceDT{\infty}{[0,t]}{U}}.
\end{equation}
\end{definition}
Note that the definition of $C$-BIBO stability does not require the output of the system itself to be a continuous function for any $u \in C\left( [0,\infty), U \right)$, nor that the corresponding solutions $(u,x,y)$ are classical.   

It has been shown that for the case that the input and the output space both are finite-dimensional (a particular case that we will refer to as \emph{SISO case} in the following\footnote{Short for \emph{single-input-single-output}.}) the three notions of BIBO stability are actually equivalent and are fully characterised by the transfer function $\mathbf{G}$ being the Laplace transform of a measure of bounded total variation, see \cite[Thm.~3.4~\&~Thm.~3.5]{a_SchwenningerWierzbaZwart24BIBO} where however the term $C$-BIBO stability was not explicitly introduced.

\section{BIBO stability under duality}
\label{sec:BIBOduality}
We now study the question raised in the introduction about when BIBO stability is preserved if the dual system is considered, i.e.\ when is $\Sigma(A^*,C^*,B^*,\overline{\mathbf{G}}(\overline{\,\cdot\,}))$ BIBO stable if we know that $\Sigma(A,B,C,\mathbf{G})$ is?

It is easy to see that for general system nodes this cannot be true. 
\begin{example}
\label{ex:dualityCounterexample}
Consider the system on the state space $X = \lpSpaceDT{2}{[0,1]}{\realNum}$ with $U = \realNum$ and $Y = X$, given by 
\begin{align}
    \Dot{x} (\xi, t) &= - \derOrd{}{\xi}{} x(\xi,t), &
    u(t) &= x(0,t), &
    y(t) &= x(\cdot, t),
\end{align}
which is a transport equation on the interval $[0,1]$ with boundary control at $0$ and the identity as observation. This is a system node where $A = - \derOrd{}{\xi}{}$ is the generator of the right-shift semigroup on the interval $[0,1]$ with $\dom(A) = \{ x \in H^1\left( [0,1], \realNum \right) : x(0) = 0 \} \subset \lpSpaceDT{2}{[0,1]}{\realNum}$, $B = \delta_0 \in W^{-1,2}([a,b],\realNum)$ and $C = \mathbb{I}$. 

The state of this system, in the sense of generalised solutions as defined above, 
can for any $u \in \lpSpacelocDT{\infty}{\posRealNum}{\realNum}$ be explicitly given as
\begin{align}
    x(\xi ,t) = \begin{cases}
        u(t - \xi) & t - \xi \geq 0 \\
        0 & t - \xi < 0
    \end{cases}
\end{align}
and thus we have in particular that the system is $L^\infty$-BIBO stable, as
\begin{align}
    \| y(t) \|^2_{\lpSpaceDT{2}{[0,1]}{\realNum}} = \int_0^1 \| x(\xi ,t) \|^2 \intd \xi \leq \| u \|^2_{\lpSpaceDT{\infty}{[0,t]}{\realNum}}.
\end{align}

The dual system is then given by
\begin{align}
    \Dot{x} (\xi, t) &= \derOrd{}{\xi}{} x(\xi,t) + u(\xi, t), &
    y(t) &= x(0,t)
\end{align}
that is by the left-shift on $[0,1]$ with boundary observation at $0$ and identity input. Again we can for any $u \in \lpSpacelocDT{\infty}{\posRealNum}{X}$ explicitly write down the system state as
\begin{align}
    x(\xi , t) = \int_{\xi + t - 1}^t u(\xi + (t - s) , s) \intd s
\end{align}
and hence have in particular
\begin{align}
    y(t) = \int_{t - 1}^t u(t - s , s) \intd s = \int_{0}^1 u(r , t-r) \intd r.
\end{align}
Consider then the sequence of functions
\begin{align}
    u_\epsilon(\xi, t) = \mathbbm{1}_{[0,1]}(t) \mathbbm{1}_{[1 - t - \frac{\epsilon}{2},1 - t + \frac{\epsilon}{2}]}(\xi)
\end{align}
for which 
while the corresponding output $y_\epsilon$ satisfies
\begin{align}
    y_\epsilon(1) &= \int_{0}^1 u_\epsilon(r , 1-r) \intd r = \int_0^1 \mathbbm{1}_{[0,1]}(1-r) \mathbbm{1}_{[r - \frac{\epsilon}{2},r + \frac{\epsilon}{2}]}(r) \intd r = 1.
\end{align}
Hence $\frac{\| y_\epsilon \|_{\lpSpaceDT{\infty}{[0,1]}{\realNum}}}{\| u_\epsilon \|_{\lpSpaceDT{\infty}{[0,1]}{\lpSpaceDT{2}{[0,1]}{\realNum}}}} \geq \frac{1}{\sqrt{\epsilon}}$
tends to infinity as $\epsilon \rightarrow 0$ and thus the system cannot be BIBO stable.
\end{example}
\begin{remark}
    That BIBO stability is not preserved in general under a duality transformation was also shown non-constructively in \cite[Rem.~6.6]{a_SchwenningerWierzbaZwart24BIBO}.
\end{remark}
For system nodes with finite-dimensional input and output spaces however, the characterisation of BIBO stability in terms of the transfer function allows to answer the question in the affirmative.
\begin{proposition}
\label{prop:finiteDimSpacesBIBODuality}
    Let $\Sigma(A, B, C, \mathbf{G})$ be a $L^\infty$-BIBO stable system node with finite-dimensional input and output spaces and assume that $A^*$ generates a $C_0$-semigroup. Then $\systemNodeStD$ is $L^\infty$-BIBO stable.
\end{proposition}
\begin{proof}
As $\Sigma(A, B, C, \mathbf{G})$ is $L^\infty$-BIBO stable, by \cite[Thm.~3.4~\&~Thm.~3.5]{a_SchwenningerWierzbaZwart24BIBO} there exists $h \in \mathcal{M}(\posRealNum, \complNum^{d_Y \times d_U})$ such that $\mathbf{G} = \laplaceTr{h}$. But then $\overline{\mathbf{G}}(\overline{\,\cdot\,}) = \laplaceTr{\overline{h}}$ and as $\overline{h} \in \mathcal{M}(\posRealNum, \complNum^{d_U \times d_Y})$ as well, the dual system node is $L^\infty$-BIBO stable too. 
 \end{proof}

\section{LILO stability}
That BIBO stability of a system does in general not imply that the dual system is BIBO stable as well, while not obvious, is also not overly surprising.
After all, duality results in systems theory usually connect \emph{different} but \emph{dual} properties of a system and its dual system.
For example, for $1 \leq p < \infty$, the dual system of any $L^p$-well-posed system --- provided it exists --- is itself $L^{p'}$-well-posed \cite[Thm.~6.2.3]{b_Staffans2005}. And for any operator that is $L^p$-observation-admissible for some $C_0$-semigroup $\mathbb{T}$, the dual operator $C^*$ is $L^{p'}$-control-admissible for the dual semigroup $\mathbb{T}^*$ provided it is strongly continuous as well \cite[Thm.~6.9]{a_WeissObservationOperators1989}. 

Thus arises the question what the corresponding dual notion to BIBO stability is that has to be considered instead. 
A natural first guess --- and as it turns out the correct one --- is the analogously defined $L^1$-to-$L^1$ input-output stability (LILO). For this reason, a characterisation of this property --- similar to the one known for BIBO stability --- would be desirable. 

We begin by defining two notions of LILO stability analogous to those of BIBO stability in Definition \ref{def:BIBOStability}.
\begin{definition}
\label{def:LILOStability}
A system node $\Sigma(A,B,C,\mathbf{G})$ is called $C^ \infty$- ($L^1$-)\emph{LILO stable} if there exists $c > 0$ such that for any classical (generalised distributional) solution $(u,x,y)$ with $u \in C^\infty_c\left( (0,\infty), U \right)$ ($u \in L^1_\loc(\posRealNum, U)$) and $x(0) = 0$ we have that $y \in L^1_\loc(\posRealNum, Y)$ and for all $t > 0$
\begin{equation}
\label{eq:LILOInequality}
    \| y \|_{\lpSpaceDT{1}{[0,t]}{Y}} \leq c \| u \|_{\lpSpaceDT{1}{[0,t]}{U}}.
\end{equation}
\end{definition}

We can then prove completely analogous to the proofs of \cite[Thm.~3.4~\&~Thm.~3.5]{a_SchwenningerWierzbaZwart24BIBO} the following full characterisation of LILO stability in the SISO case.
\begin{theorem}
\label{thm:characterizationLILO}
    Let $\Sigma(A,B,C,\mathbf{G})$ be a system node with finite dimensional input and output spaces. Then the following are equivalent:
    \begin{enumerate}
        \item $\Sigma(A,B,C,\mathbf{G})$ is $C^\infty$-LILO stable.
        \item $\Sigma(A,B,C,\mathbf{G})$ is $L^1$-LILO stable.
        \item There exists a measure of bounded total variation $h \in \mathcal{M}(\posRealNum, \boundedOp{U}{Y})$ such that $\mathbf{G} = \laplaceTr{h}$ on $\complNum_{\omega(\mathbb{T})}$.
    \end{enumerate}
\end{theorem}
\begin{proof}
    The proof is completely analogous to the proofs of Propositions 3.6, 3.8 and 3.9 in \cite{a_SchwenningerWierzbaZwart24BIBO}, noting that the cited results \cite[Thm.~3.6.1(i)]{b_GripenbergLondenStaffans1990} and \cite[Thm.~I.3.16 \& Thm.~I.3.19]{b_SteinWeiss2016} can likewise be applied to the $L^1$-case and all necessary norm bounds can be found for the respective $L^1$-norms as well. 
 \end{proof}

\begin{remark}
\label{rem:BIBOLILOimplypIpO}
    The arguments from the proof above (respecitvely from \cite{a_SchwenningerWierzbaZwart24BIBO}) can in fact be extended to show that existence of a measure of bounded total variation $h \in \mathcal{M}(\posRealNum, \complNum^{d_Y \times d_U})$ such that $\laplaceTr{h} = \mathbf{G}$ on the half-plane $\complNum_{\omega(\mathbb{T})}$, implies that for any $u \in \lpSpaceDT{p}{\posRealNum}{U}$ with $p \in [1,\infty]$ we have that $y \in  \lpSpaceDT{p}{\posRealNum}{Y}$ and
    \begin{equation}
        \|y\|_{\lpSpaceDT{p}{\posRealNum}{Y}} \leq \| h \|_{\mathcal{M}} \|u\|_{\lpSpaceDT{p}{\posRealNum}{U}},
    \end{equation}
    i.e.\ $C^\infty$-LILO and $C^\infty$-BIBO stability both imply any $L^p$-to-$L^p$ stability defined analogously in the sense of Definitions \ref{def:BIBOStability} and \ref{def:LILOStability}.
\end{remark}

Together with the characterisation of $L^\infty$-BIBO stability from \cite{a_SchwenningerWierzbaZwart24BIBO}, Theorem~\ref{thm:characterizationLILO} immediately implies the equivalence of BIBO and LILO stability in the SISO case.
\begin{corollary}
\label{cor:BIBOLILOequivalence}
    Let $\Sigma(A,B,C,\mathbf{G})$ be a system node with finite dimensional input and output spaces. Then the following are equivalent:
    \begin{enumerate}
        \item $\Sigma(A,B,C,\mathbf{G})$ is $L^\infty$-BIBO stable.
        \item $\Sigma(A,B,C,\mathbf{G})$ is $L^1$-LILO stable.
    \end{enumerate}
    If furthermore  $A^*$ generates a $C_0$-semigroup the above is also equivalent to
    \begin{enumerate}
    \setcounter{enumi}{2}
        \item $\systemNodeStD$ is $L^\infty$-BIBO stable.
        \item $\systemNodeStD$ is $L^1$-LILO stable.
    \end{enumerate}
\end{corollary}
\begin{remark}
    The statement of the previous corollary is not  unexpected in light of results on bounded extendability of convolution operators to different function spaces, see for example \cite{a_UnserNoteOnBiboStability} for a treatment of such questions aimed at BIBO stability in particular. However, just as was argued in \cite[Sec.~3.3.]{a_SchwenningerWierzbaZwart24BIBO}, the main point of the proof leading up to Corollary \ref{cor:BIBOLILOequivalence} is to show that these bounded extensions coincide with the actual input-output mapping resulting from the generalised solution concept.
\end{remark}

\section{A sufficient criterion for LILO stability}
Theorem \ref{thm:characterizationLILO} provides a characterisation of both $C^\infty$- as well as $L^1$-LILO stability, however it is restricted to the SISO case. While so far we cannot give a similar characterisation for the general case allowing for infinite-dimensional input and output spaces, one can --- in the same spirit as it was done in \cite{a_SchwenningerWierzbaZwart24BIBO} for BIBO stability --- derive sufficient conditions for LILO stability in this less restrictive setting.

For this, note first that the equivalence of $C^\infty$- and $L^1$-LILO stability can be extended beyond the SISO case.
\begin{proposition}
\label{prop:CLILOLLILOequiv}
    $\systemNode{A}{B}{C}{\mathbf{G}}$ is $C^\infty$-LILO stable if and only it is $L^1$-LILO stable.  
\end{proposition}
\begin{proof}
    The direction from $L^1$-LILO stability to $C^\infty$-LILO stability is clear.

    For the other direction, let $u \in \lpSpaceDT{1}{\posRealNum}{U}$ and $(u_k)_{k\geq 0} \subset C^\infty_c(\posRealNum, U)$ a sequence such that $u_k \rightarrow u$ in $\lpSpaceDT{1}{\posRealNum}{U}$ as $k \rightarrow \infty$. Then consider the classical solutions $(u_k, x_k, y_k)$ and the generalised solution $(u,x,y)$.

    First we see that $y_k \rightarrow \widetilde{y}$ for some $\widetilde{y} \in \lpSpaceDT{1}{\posRealNum}{Y}$, as $u_m - u_n \in C^\infty_c(\posRealNum, U)$ for any $m,n > 0$ and thus by  the assumed $C^\infty$-LILO stability $\|y_m - y_n\|_{\lpSpaceDT{1}{\posRealNum}{Y}} \leq c \|u_m - u_n\|_{\lpSpaceDT{1}{\posRealNum}{U}}$ for some uniform constant $c > 0$ making $(y_k)$ a Cauchy sequence. Furthermore, by continuity we have that $\|\widetilde{y}\|_{\lpSpaceDT{1}{\posRealNum}{Y}} \leq c \|u\|_{\lpSpaceDT{1}{\posRealNum}{U}}$.

    It remains to show that $\widetilde{y} = y$, that is that the limit of the $y_k$ coincides with the distributionally defined output $y$ of the generalised solution. For this, define first $\left[ \begin{smallmatrix}
        K_k \\ L_k
    \end{smallmatrix} \right] := \int_0^\cdot (\cdot - s) \left[ \begin{smallmatrix}
            x_k(s) \\ u_k(s)
        \end{smallmatrix} \right] \intd s$ and $\left[ \begin{smallmatrix}
        K \\ L
    \end{smallmatrix} \right] := \int_0^\cdot (\cdot - s) \left[ \begin{smallmatrix}
            x(s) \\ u(s)
        \end{smallmatrix} \right] \intd s$. Then observe that we have
        \begin{equation}
        \label{eq:localUnifLkConv}
        \begin{split}
            \left\| L - L_k \right\|_{\lpSpaceDT{\infty}{[0,T]}{U}} &= \left\| \int_0^\cdot (\cdot - s)
            (u(s) - u_k(s) ) \intd s  \right\|_{\lpSpaceDT{\infty}{[0,T]}{U}} \\ &\leq T \left\| u - u_k  \right\|_{\lpSpaceDT{1}{[0,T]}{U}}.
            \end{split}
        \end{equation}
        Furthermore $K_k, K \in W^{2,1}\left( \posRealNum, U \right)$ and thus by \cite[Lem.~4.7.8]{b_Staffans2005} we have
        \begin{equation*}
            \dot{K}(t) - \dot{K}_k(t) = A_{-1} \left( K(t) - K_k(t) \right) + B \left( L(t) - L_k(t) \right)
        \end{equation*}
        and for $s \in \rho(A)$ we can find
    \begin{align*}
            K(t) - K_k(t) = &- \resolvent{s}{A_{-1}} \int_0^t x(s) - x_k(s) \intd s  + s \resolvent{s}{A_{-1}} \left( K(t) - K_k(t) \right) \\ &+ \resolvent{s}{A_{-1}} B \left( L(t) - L_k(t) \right).
        \end{align*}
    Thus we have the estimate
    \begin{equation}
    \label{eq:KkIntermEstimate}
    \begin{split}
            \left\| K(t) - K_k(t) \right\|_X \leq &- \left\| \resolvent{s}{A_{-1}} \right\|_{\boundedOp{X_{-1}}{X}} \left\| \int_0^t x(s) - x_k(s) \intd s \right\|_{X_{-1}} \\ &+ |s| \left\| \resolvent{s}{A_{-1}} \right\|_{\boundedOp{X_{-1}}{X}} \left\| K(t) - K_k(t) \right\|_{X_{-1}} \\ &+ \left\| \resolvent{s}{A_{-1}} \right\|_{\boundedOp{X_{-1}}{X}} \| B \|_{\boundedOp{U}{X_{-1}}} \left\| L(t) - L_k(t) \right\|_U.
    \end{split}
    \end{equation}
    Now, as $B \in \boundedOp{U}{X_{-1}}$ it is in particular $L^\infty$-control-admissible for the state space $X_{-1}$ and thus from $K(t) - K_k(t) = \int_0^t \semiGroupEx{t-s} B (L(s) - L_k(s) ) \intd s$ we can conclude for any $T > 0$ and $t \in [0,T]$ that 
    \begin{equation}
        \| K(t) - K_k(t) \|_{X_{-1}} \leq c_{T} \| L -\ L_k \|_{\lpSpaceDT{\infty}{[0,T]}{U}} \leq c_T T \| u - u_k \|_{\lpSpaceDT{1}{[0,T]}{U}}
    \end{equation}
    for some constant $c_T > 0$. Similarly, as $B$ is also $L^1$-control-admissible for the state space $X_{-1}$ as a bounded operator, we find for some constant $\widetilde{c}_T > 0$ that
    \begin{equation}
        \left\| \int_0^t x(s) - x_k(s) \intd s \right\|_{X_{-1}} \leq \widetilde{c}_{T} T \| u - u_k \|_{\lpSpaceDT{1}{[0,T]}{U}}.
    \end{equation}
    Then with these observations Equation~\eqref{eq:KkIntermEstimate} implies for any $T > 0$ that $\left\| K(t) - K_k(t) \right\|_X \\ \leq \kappa_T \| u - u_k \|_{\lpSpaceDT{1}{[0,T]}{U}}$ for some constant $\kappa_T > 0$ and hence from this and Equation~\eqref{eq:localUnifLkConv} we find that $\left[ \begin{smallmatrix}
        K_k \\ L_k
    \end{smallmatrix} \right] \rightarrow \left[ \begin{smallmatrix}
        K \\ L
    \end{smallmatrix} \right]$ locally uniformly in $X \times U$.

    Further, considering $(u_k, x_k, y_k)$ as a generalised distributional solutions, we find that by the $C^\infty$-LILO stability we have
    \begin{align*}
    &\left\| C\& D \begin{bmatrix}
        K_m \\ L_m
    \end{bmatrix} - C\& D \begin{bmatrix}
        K_n \\ L_n
    \end{bmatrix}  \right\|_{\lpSpaceDT{\infty}{[0,T]}{Y}} = \left\| \int_0^\cdot (\cdot - s)
            (y_m(s) - y_n(s) ) \intd s  \right\|_{\lpSpaceDT{\infty}{[0,T]}{Y}}\\ &\qquad\leq T \left\| y_m - y_n  \right\|_{\lpSpaceDT{1}{[0,T]}{Y}} \leq c T \left\| u_m - u_n  \right\|_{\lpSpaceDT{1}{[0,T]}{U}}.
    \end{align*}
    This renders $\CandD \left[ \begin{smallmatrix}
        K_k \\ L_k
    \end{smallmatrix} \right]$ a locally uniform Cauchy sequence and therefore also locally uniformly convergent. Hence $\left[ \begin{smallmatrix}
        K_k \\ L_k
    \end{smallmatrix} \right]$ converges locally uniformly in the $C \& D$ graph norm

    Considering now the action of $\widetilde{y}$ on a test function $\varphi \in C^\infty_c(\posRealNum, Y')$ and using the locally uniform convergence and the continuity of $\CandD$ in its graph norm we find that 
    \begin{align*}
        \innerProd{\varphi}{\widetilde{y}} = \lim_{k\rightarrow\infty} \innerProd{\varphi}{y_k} 
        &= \lim_{k\rightarrow\infty} \innerProd{\varphi''}{\int_0^\cdot (\cdot - s) C\& D \begin{bmatrix}
            x_k(s) \\ u_k(s)
        \end{bmatrix} \intd s}
        \\
        &= \lim_{k\rightarrow\infty} \innerProd{\varphi''}{C\& D \begin{bmatrix}
            K_k(s) \\ L_k(s)
        \end{bmatrix}} = \innerProd{\varphi''}{\lim_{k\rightarrow\infty} C\& D \begin{bmatrix}
            K_k(s) \\ L_k(s)
        \end{bmatrix}} 
        \\
        &= \innerProd{\varphi''}{C\& D \begin{bmatrix}
            K(s) \\ L(s)
        \end{bmatrix}} 
        = \innerProd{\varphi}{y}.
    \end{align*}   
 \end{proof}

To derive a sufficient condition for LILO stability we will make use of the following characterisation of $L^1$-observation-admissible operators.
This result, in the special case that $C = A$, is closely related to the idea of $L^1$-maximal regularity \cite{a_JacobSchwenningerWintermayr2022} and the following proof is based upon the approach used in \cite{a_KaltonPortal08MaximalRegularity}, where, however, the notion of admissible operators was not used explicitly. 
\begin{proposition}
    \label{prop:admissibilityMaxReg}
    Let $0 \in \rho(A)$. Then an operator $C \in \boundedOp{X_1}{Y}$ is infinite-time $L^1$-observation-admissible if and only if there exists a $c>0$ such that for all $u \in \lpSpaceDT{1}{\posRealNum}{X}$ one has that
    \begin{equation}
    \label{eq:maxRegProperty}
        \left\| C \int_0^\cdot \semiGroup{\cdot - s} u(s) \intd s \right\|_{\lpSpaceDT{1}{\posRealNum}{Y}} \leq c \left\| u \right\|_{\lpSpaceDT{1}{\posRealNum}{X}}.
    \end{equation}
\end{proposition}
\begin{proof}
        Let $C \in \boundedOp{X_1}{Y}$ be infinite-time $L^1$-observation-admissible, i.e.\ there exists $\kappa > 0$ such that for any $x \in \dom(A)$
    \begin{align*}
        \| C \semiGroup{\cdot} x \|_{\lpSpaceDT{1}{\posRealNum}{Y}} \leq \kappa \| x \|_X.
    \end{align*}
    Then for any $0 < h < \infty$ and any sequence $(x_k)_{k>0}$ in $D(A)$ with $\sum_{k=1}^\infty \|x_k \|_X < \infty$ define for $n > 0$
    \begin{align*}
        v_n[x] = \sum_{k=1}^{n}  C \inv{A} \semiGroup{(n-k)h} \left( \semiGroup{h} - I \right) x_k.
    \end{align*}
    For this sequence we then find that
    \begin{align*}
        &\sum_{n=1}^{\infty} \left\| v_n[x] \right\|_Y =
        \sum_{n=1}^{\infty} \left\| \sum_{k=1}^{n}  C \inv{A} \semiGroup{(n-k)h} \left( \semiGroup{h} - I \right) x_k \right\|_Y \\
        &\quad \leq
        \sum_{n=1}^{\infty}  \sum_{k=1}^{n} \left\|  C \inv{A} \semiGroup{(n-k)h} \left( \semiGroup{h} - I \right) x_k \right\|_Y \\
        &\quad =
        \sum_{k=1}^{\infty}  \sum_{j=1}^{\infty} \left\|  C \inv{A} \semiGroup{(j-1)h} \int_0^h A \semiGroup{\tau} x_k \intd \tau \right\|_Y \\
        &\quad \leq
        \sum_{k=1}^{\infty}  \int_0^\infty \left\|  C \semiGroup{t} x_k \right\|_Y \intd t  \leq \kappa \sum_{j=1}^\infty \| x_k \|_X.
    \end{align*}
        Consider then for any such sequence $(x_k)_{k>0}$ as above the step function $f = \sum_{k=1}^\infty x_k \idMatrix{((k-1)h, kh)}$ and define for $t\geq0$
    \begin{align*}
        F(t) = \int_0^t C \semiGroup{t-r} f(r) \intd r.
    \end{align*}
    Then for any $0 \leq \tau < h$ and $n \geq 1$ we find that
    \begin{align*}
        F((n-1)h + \tau) &=  
         \sum_{k=1}^{n-1} \int_0^h C \semiGroup{(n-k)h + \tau - r} x_k \intd r + \int_0^\tau C \semiGroup{\tau - r} x_n \intd r.
    \end{align*}
    For the first term we find that
    \begin{align*}
        &\sum_{k=1}^{n-1} \int_0^h C \semiGroup{(n-k)h + \tau - r} x_k \intd r  \\
        &= \sum_{k=1}^{n-1} C \inv{A} \semiGroup{(n - 1 -k)h} \int_0^h A \semiGroup{h - r} \semiGroup{\tau} x_k \intd r \\
        &= \sum_{k=1}^{n-1} C \inv{A} \semiGroup{(n - 1 -k)h} \left( \semiGroup{h} - I \right) \semiGroup{\tau} x_k \intd r = v_{n-1}[\semiGroup{\tau} x]
    \end{align*}
    and for the second one that
    \begin{align*}
        \int_0^\tau C \semiGroup{\tau - r} x_n &= C \inv{A} \int_0^\tau A \semiGroup{\tau - r} x_n \intd r = C \inv{A} \left( \semiGroup{\tau} - I \right) x_n.
    \end{align*}
    Then for $h < h_0$ sufficiently small we have $\| \semiGroup{\tau} \| \leq \sup_{t \in [0,h_0]} \| \semiGroup{t} \| =: M < \infty$ independent of $h$.
    We then find that
    \begin{align*}
        &\int_0^\infty \| F(t) \| \intd t = \sum_{n=1}^\infty \int_0^h \| F((n-1)h + \tau) \| \intd \tau \\
        &\quad\leq \int_0^h \sum_{n=1}^\infty \| v_{n-1}[\semiGroup{\tau} x] \|_Y \intd \tau + h \| C \inv{A} \| \left( M + 1 \right) \sum_{n=1}^\infty \| x_k \|_X \\
        &\quad\leq \int_0^h \kappa \sum_{n=1}^\infty \| \semiGroup{\tau} x_n \|_X + \| C \inv{A} \| \left( M + 1 \right) \| f \|_{\lpSpaceDT{1}{\posRealNum}{X}} \\
        &\quad\leq \left( \kappa M + \| C \inv{A} \| \left( M + 1 \right) \right) \| f \|_{\lpSpaceDT{1}{\posRealNum}{X}}.
    \end{align*}
    The density of these step functions in $\lpSpaceDT{1}{\posRealNum}{X}$ then implies Equation \eqref{eq:maxRegProperty}. 

    For the converse, assume that Equation \eqref{eq:maxRegProperty} holds. Then consider for $n > 0$ and $x \in \dom(A)$ the functions $u_n := n \idMatrix{[0,\frac{1}{n}]}(\cdot) \semiGroup{\cdot} x \in \lpSpaceDT{1}{\posRealNum}{X}$. Then 
    \begin{align*}
        \left\| u_n \right\|_{\lpSpaceDT{1}{\posRealNum}{X}} = \left\| n \idMatrix{[0,\frac{1}{n}]}\semiGroup{\cdot} x \right\|_{\lpSpaceDT{1}{\posRealNum}{X}} &= n \int_0^{\frac{1}{n}} \left\| \semiGroup{s} x \right\|_X \intd s \\ &\leq \sup_{s \in [0,\frac{1}{n}]} \| \semiGroup{s} \| \|x\|_X
    \end{align*}
    while for the left hand side of Equation \eqref{eq:maxRegProperty} we find that
    \begin{align*}
        \left\| C \int_0^\cdot \semiGroup{\cdot - s} u_n(s) \intd s \right\|_{\lpSpaceDT{1}{\posRealNum}{Y}} &= \left\| C \int_0^\cdot \semiGroup{\cdot - s} n \idMatrix{[0,\frac{1}{n}]}(s) \semiGroup{s} x \intd s \right\|_{\lpSpaceDT{1}{\posRealNum}{Y}} \\
        &\geq \left\| C \semiGroup{\cdot} x \right\|_{\lpSpaceDT{1}{[\frac{1}{n}, \infty)}{Y}}.
    \end{align*}
    Now as there is $\rho > 0$ and $\omega > \omega_{\mathbb{T}}$ such that $\| \semiGroup{s} \| \leq \rho e^{\omega s}$ for $s \geq 0$, we thus find that Equation \eqref{eq:maxRegProperty} implies for any $x \in \dom(A)$ and any $n > 0$ that
    \begin{equation}
        \left\| C \semiGroup{\cdot} x \right\|_{\lpSpaceDT{1}{[\frac{1}{n},\infty)}{Y}} \leq c \rho \max\left\lbrace 1, e^{\omega} \right\rbrace \| x \|_X,
    \end{equation}
    and thus in particular also
    \begin{equation}
        \left\| C \semiGroup{\cdot} x \right\|_{\lpSpaceDT{1}{\posRealNum}{Y}} \leq c \rho \max\left\lbrace 1, e^{\omega} \right\rbrace \| x \|_X.
    \end{equation}
    Hence $C$ is infinite-time $L^1$-observation-admissible. 
\end{proof}
\begin{remark}
    Note that the assumption $0 \in \rho(A)$ is not necessary in the case $C = A$ as discussed in \cite{a_KaltonPortal08MaximalRegularity}. We believe that it should be similarly superfluous here, but so far eliminating the assumption from the proof remains unclear.
\end{remark}

Using Proposition~\ref{prop:CLILOLLILOequiv} we find that a corollary of this characterisation is then the following dual analogon to Lemma~6.1 from \cite{a_SchwenningerWierzbaZwart24BIBO}.
\begin{corollary}
\label{cor:CObsAdmLILO}
    Let $\systemNode{A}{B}{C}{\mathbf{G}}$ be a system node with $0 \in \rho(A)$, $B \in \boundedOp{U}{X}$ and $C$ infinite-time $L^1$-observation-admissible. Then $\systemNode{A}{B}{C}{\mathbf{G}}$ is $L^1$-LILO stable.
\end{corollary}
\begin{proof}
    It is clear that it is sufficient to show this for $B = \mathbb{I}$ and $U = X$. In that case, for any $u \in C^\infty_c(\posRealNum, X)$ the output is given by
    \begin{align*}
        y(t) = C \& D \begin{bmatrix}
            x(t) \\ u(t)
        \end{bmatrix} 
        &= C x(t) + \left( \mathbf{G}(s) - C \resolvent{s}{A} \right) u(t).
    \end{align*}
    By Proposition \ref{prop:admissibilityMaxReg} there exists $c > 0$ such that $\left\| C x \right\|_{\lpSpaceDT{1}{\posRealNum}{Y}} \leq c \left\| u \right\|_{\lpSpaceDT{1}{\posRealNum}{X}}$ for all $u \in C^\infty_c(\posRealNum, X)$ and we have $\mathbf{G}(s) - C \resolvent{s}{A} \in \boundedOp{X}{Y}$, so that we can conclude that $\left\| y \right\|_{\lpSpaceDT{1}{\posRealNum}{Y}} \leq (c + \| \mathbf{G}(s) - C \resolvent{s}{A} \|) \left\| u \right\|_{\lpSpaceDT{1}{\posRealNum}{X}}$.
    
    Hence $\systemNode{A}{B}{C}{\mathbf{G}}$ is $C^\infty$-LILO stable and, by Proposition \ref{prop:CLILOLLILOequiv}, also $L^1$-LILO stable.
 \end{proof}

\begin{example}
\label{ex:ShiftLILO}
    Consider again the dual system from Example \ref{ex:dualityCounterexample} which was given by the left-shift on $[0,1]$ with boundary observation at $0$ and identity input, that is by the equations 
    \begin{align}
        \Dot{x} (\xi, t) &= \derOrd{}{\xi}{} x(\xi,t) + u(\xi, t), &
        y(t) &= x(0,t).
    \end{align}
    This is a system node with $A = \derOrd{}{\xi}{}$ the generator of the left-shift semigroup on the interval $[0,1]$ and $\dom(A) = \left\{ x \in H^1\left( [0,1], \realNum \right) : x(1) = 0 \right\} \subset \lpSpaceDT{2}{[0,1]}{\realNum}$, $B = \mathbb{I} \in \boundedOpSelf{\lpSpaceDT{2}{[0,1]}{\realNum}}$ and $C = \cdot |_{0}$.

    Then for $x \in \dom(A)$ and $t\geq0$ we have
    \begin{align*}
        C \semiGroup{t} x = \begin{cases}
            x(t) & 0 \leq t \leq 1 \\
            0 & t > 1
        \end{cases}
    \end{align*}
    and thus in particular
    \begin{align*}
        \| C \semiGroup{\cdot} x \|_{\lpSpaceDT{1}{\posRealNum}{\realNum}} = \| x \|_{\lpSpaceDT{1}{[0,1]}{\realNum}} \leq \| x \|_{\lpSpaceDT{2}{[0,1]}{\realNum}}
    \end{align*}
    showing that the operator $C$ is infinite-time $L^1$-observation-admissible. As the semigroup generated by $A$ is furthermore exponentially stable and thus $0 \in \rho(A)$, we can apply Corollary \ref{cor:CObsAdmLILO} to conclude that this system is $L^1$-LILO stable.
\end{example}

\begin{remark}
\label{rem:BIBOLILOnonEquiv}
    The study of the system discussed above shows that the equivalence of BIBO and LILO stability that holds in the SISO case by Corollary \ref{cor:BIBOLILOequivalence} cannot be true in general as it is $L^1$-LILO stable by Example \ref{ex:ShiftLILO} but not $L^\infty$-BIBO stable by Example \ref{ex:dualityCounterexample}.
\end{remark}
A more extensive dual result to Corollary~\ref{cor:CObsAdmLILO} for control-admissible operators (and a refinement of \cite[Lem.~6.1]{a_SchwenningerWierzbaZwart24BIBO}) can be derived using the following lemma.
\begin{lemma}
\label{lem:convergenceXWeakDual}
    Let $(x_k)_{k \in \natNum} \subset X$ be a sequence in $X$ such that $\| x_k \|_X \leq c$ for some $c > 0$ and all $k \in \natNum$ and which in $X_{-1}$ converges to $x \in X$. 
    Then $\| x \|_X \leq \kappa c$, where $\kappa = \liminf_{\lambda \rightarrow \infty} \| \resolvent{\lambda}{A} \| < \infty$.
\end{lemma}
\begin{proof}
    The convergence $x_k \rightarrow x$ in $X_{-1}$ means that for $s \in \rho(A)$ we have $\resolvent{s}{A} x_k \\ \rightarrow \resolvent{s}{A} x$ in $X$.
    Then we also have for any bounded linear functional $f \in X^*$ that $f \left( \resolvent{s}{A} x_k \right) \rightarrow f\left( \resolvent{s}{A} x \right)$ and thus also $\left( (s\mathbb{I}-A)^{-\ast} f \right) \left(  x_k \right) \rightarrow \left( (s\mathbb{I}-A)^{-\ast} f \right) \left(  x \right)$.

    Now, as $(s\mathbb{I}-A)^{-\ast} X^* = \dom(A^*) \subseteq X^\odot$ densely \cite[Thm.~1.3.1]{b_vanNeerven06_AdjointSemigroup}, we can find for any $f \in X^\odot$ a sequence $(f_k)_{k \in \natNum} \subseteq \dom (A^*)$ such that $f_k \rightarrow f$ in $X^*$. Then
    \begin{align}
        | f(x) - f(x_k) | &\leq | f(x) - f_n(x) | + | f_n(x) - f_n(x_k) | + | f_n(x_k) - f(x_k) | \\
        &\leq \| f - f_n \|_{X^*} \| x \|_X + | f_n(x) - f_n(x_k) | + \| f - f_n \|_{X^*} \| x_k \|_X \\
        &\leq \| f - f_n \|_{X^*} \left( \| x \|_X + c \right) + | f_n(x) - f_n(x_k) |,
    \end{align}
    for arbitrary $n \in \natNum$ and thus we find $f(x_k) \rightarrow f(x)$.

    By \cite[Thm.~14.2.1]{b_HillePhillips57_FunctionalAnalysis} we have
    \begin{equation}
        \frac{1}{\kappa} \| x \|_X \leq \sup_{\substack{f \in X^\odot \\ \| f \|_{X^*} = 1}} | f(x) |
        = \sup_{\substack{f \in X^\odot \\ \| f \|_{X^*} = 1}} \left| \lim_{k \rightarrow \infty} f(x_k) \right| = \sup_{\substack{f \in X^\odot \\ \| f \|_{X^*} = 1}} \lim_{k \rightarrow \infty} \left| f(x_k) \right|,
    \end{equation}
    and as again by \cite[Thm.~14.2.1]{b_HillePhillips57_FunctionalAnalysis} we have $|f(x_k)| \leq \| x_k \|_X \leq c$, we find $\| x \|_X \leq \kappa c$. 
 \end{proof}
The following result shows that if the output equals the state, then BIBO stability and admissibility with respect to continuous functions are equivalent.
\begin{proposition}\label{prop:CBIBOadm}
    A system node $\systemNode{A}{B}{\mathbb{I}}{\resolvent{\cdot}{A} B}$ is $C^\infty$-BIBO stable if and only if $B$ is infinite-time $C$-control-admissible.
\end{proposition}
\begin{proof}
    Assume that $\systemNode{A}{B}{\mathbb{I}}{\resolvent{\cdot}{A} B}$ is $C^\infty$-BIBO stable. Observe first that $C\& D \left[ \begin{smallmatrix}
        x \\ u
    \end{smallmatrix} \right] = x$ and thus we can treat the distributional derivatives in Equation \eqref{eq:distributionalOutputDef} as ordinary derivatives and find $y = x$ for any generalised solution. 
    
    Consider then a classical solution with $u \in C^\infty_c(\posRealNum, U)$ and $x(0) = 0$. Then the BIBO stability implies for any $t \geq 0$ that
    \begin{equation}
    \label{eq:almostAdmissibilityCondition}
        \| x \|_{\lpSpaceDT{\infty}{[0,t]}{X}} = \left\| \int_0^t \semiGroup{t-s} B u(s) \intd s \right\|_{\lpSpaceDT{\infty}{[0,t]}{X}} \leq c \| u \|_{\lpSpaceDT{\infty}{[0,t]}{U}}.
    \end{equation}
    for some $c > 0$ and as by \cite[Thm.~3.8.2]{b_Staffans2005} we have $x \in C(\posRealNum, X)$ we even find 
    \begin{equation}
        \left\| \int_0^t \semiGroup{t-s} B u(s) \intd s \right\|_{X} \leq c \| u \|_{\lpSpaceDT{\infty}{[0,t]}{U}}.
    \end{equation}
    By continuity this extends to all $u \in \overline{C^\infty_c(\posRealNum, U)}^{\lpSpaceDT{\infty}{\posRealNum}{U}} = C_0(\posRealNum, U)$, i.e.\ all continuous inputs vanishing for $t = 0$.

    Now take any $u \in C^1([0,\infty), U)$. Then $B u \in W^{1,1}_\loc(\posRealNum, X_{-1})$ and by \cite[Thm.~3.8.2]{b_Staffans2005} we thus have $x := \int_0^\cdot \semiGroup{\cdot - s} B u(s) \intd s \in C(\posRealNum, X)$.
    
    There is always a sequence $(u_k)_{k \in \natNum}$ in $C_0(\posRealNum, U)$ such that $u_k \rightarrow u$ in $\lpSpacelocDT{1}{\posRealNum}{U}$ and $\| u_k \|_{\lpSpaceDT{\infty}{[0,t]}{[U]}} \leq \| u \|_{\lpSpaceDT{\infty}{[0,t]}{[U]}}$ for any $k \in \natNum$  and $t \geq 0$. By the first part of this proof we have for any $k \in \natNum$ that $x_k := \int_0^\cdot \semiGroup{\cdot - s} B u_k(s) \intd s \in C(\posRealNum, X)$ and $\| x_k \|_{\lpSpaceDT{\infty}{[0,t]}{[U]}} \leq c \| u_k \|_{\lpSpaceDT{\infty}{[0,t]}{[U]}} \leq c \| u \|_{\lpSpaceDT{\infty}{[0,t]}{[U]}}$.

    Furthermore, as $B$ is bounded from $U$ to $X_{-1}$ we have $x_k(t) \rightarrow x(t)$ in $X_{-1}$ and Lemma~\ref{lem:convergenceXWeakDual} then implies that $\| x \|_X \leq \kappa c \| u \|_{\lpSpaceDT{\infty}{[0,t]}{[U]}}$. Again by continuity this inequality then extends to all $u \in \overline{C^1(\posRealNum, U)}^{\lpSpaceDT{\infty}{\posRealNum}{U}} = C(\posRealNum, U)$, i.e.\ $B$ is $C$-control-admissible.

    The converse direction is clear. 
 \end{proof}

\begin{remark}
    It is easy to see that $B$ being infinite-time $C$-control-admissible even implies $C$-BIBO stability of $\systemNode{A}{B}{\mathbb{I}}{\resolvent{\cdot}{A} B}$. On the other hand, it is a natural question to ask whether the analogous result of Proposition \ref{prop:CBIBOadm} holds for $L^{\infty}$.
\end{remark}
\section{Duality between BIBO and LILO stability}
Equipped with the definition of LILO stability we can now show that it indeed serves as the dual notion of BIBO stability when going from a system node $\systemNodeSt$ to its dual $\systemNodeStD$ and vice versa.

To do this we will make use of the following two  lemmas.
\begin{lemma}
\label{lem:classicalSolDualityIntegrals}
    Let $\systemNode{A}{B}{C}{\mathbf{G}}$ be a system node such that $A^*$ generates a $C_0$-semigroup. Then for any any classical solution $(u,x,y)$ of $\systemNode{A}{B}{C}{\mathbf{G}}$ with $x(0) = 0$, any classical solution $(u^d,x^d ,y^d)$ of $\systemNodeStD$ with $x^d(0) = 0$ and any $h > 0$ we have that
    \begin{equation*}
        \int_0^h \innerProd{y(t)}{u^d(h-t)}_{Y,Y^*} \intd t = \int_0^h \innerProd{u(t)}{y^d(h-t)}_{U,U^*} \intd t.
    \end{equation*}
\end{lemma}
\begin{proof}
    Let $\systemNode{A}{B}{C}{\mathbf{G}}$ be as in the lemma and denote by $S = \left[ \begin{smallmatrix}
        A \,\, B \\ C \& D 
    \end{smallmatrix} \right]: \dom(C \& D) \subseteq X \times U \rightarrow X \times Y$ its system operator. Then by \cite[Lem.~6.2.14]{b_Staffans2005} the system operator of the dual system node $\systemNodeStD$ is given by the Banach space adjoint $S^*$.

    Let now $(u,x,y)$ be any classical solution of $\systemNode{A}{B}{C}{\mathbf{G}}$ with $x(0) = 0$ and let $(u^d,x^d ,y^d)$ be any classical solution of $\systemNodeStD$ with $x^d(0) = 0$. Then for $0 \leq t \leq h$ we have
    \begin{equation}
        \begin{bmatrix}
            x(t) \\ u(t)
        \end{bmatrix} \in \dom(S) \quad \andMath \quad \begin{bmatrix}
            x^d(h-t) \\ u^d(h-t)
        \end{bmatrix} \in \dom(S^*) 
    \end{equation}
    and thus in particular
    \begin{equation}
    \begin{split}
        &\innerProd{S \begin{bmatrix}
            x(t) \\ u(t)
        \end{bmatrix}}{\begin{bmatrix}
            x^d(h-t) \\ u^d(h-t)
        \end{bmatrix}}_{\left[ \begin{smallmatrix}
            X \\ Y
        \end{smallmatrix}\right],\left[ \begin{smallmatrix}
            X^* \\ Y^*
        \end{smallmatrix}\right]}  = \innerProd{\begin{bmatrix}
            x(t) \\ u(t)
        \end{bmatrix}}{S^* \begin{bmatrix}
            x^d(h-t) \\ u^d(h-t)
        \end{bmatrix}}_{\left[ \begin{smallmatrix}
            X \\ U
        \end{smallmatrix}\right],\left[ \begin{smallmatrix}
            X^* \\ U^*
        \end{smallmatrix}\right]}.
    \end{split}
    \end{equation}
    Using that $S \left[ \begin{smallmatrix}
        x(t) \\ u(t)
    \end{smallmatrix} \right] = \left[ \begin{smallmatrix}
        \Dot{x}(t) \\ y(t)
    \end{smallmatrix} \right]$ for classical solutions and writing out the components of the duality products gives
    \begin{equation*}
        \begin{split}
            &\innerProd{\Dot{x}(t)}{x^d(h-t)}_{X,X^*} + \innerProd{y(t)}{u^d(h-t)}_{Y,Y^*}\\ &\qquad\qquad\qquad\qquad  = \innerProd{x(t)}{\Dot{x}^d(h-t)}_{X,X^*} + \innerProd{u(t)}{y^d(h-t)}_{U,U^*},
        \end{split}
    \end{equation*}
    which, upon integrating from $t=0$ to $t=h$ and employing the initial conditions yields
    \begin{equation*}
        \int_0^h \innerProd{y(t)}{u^d(h-t)}_{Y,Y^*} \intd t = \int_0^h \innerProd{u(t)}{y^d(h-t)}_{U,U^*} \intd t.
    \end{equation*} 
 \end{proof}

\begin{lemma}
\label{lem:smoothSolDerivativeSol}
    Let $(u,x,y)$ be a classical solution of the system node $\systemNodeSt$ with $x(0) = 0$ and $u \in C_c^\infty(\posRealNum, U)$.
    Then $(u'', x'', y'')$ is also a classical solution of $\systemNodeSt$ with $x''(0) = 0$.
\end{lemma}
\begin{proof}
As $u'' \in C_c^\infty(\posRealNum, U)$, there exists a classical solution $(u'',\widetilde{x}, \widetilde{y})$.
Since
    \begin{equation*}
        \begin{split}
            x'(t) &= \mathbb{T}(t) B u(0) + \int_0^t \mathbb{T}(s) B u'(t-s) \intd s = \int_0^t \mathbb{T}(s) B u'(t-s) \intd s \\
            x''(t) &= \int_0^t \mathbb{T}(s) B u''(t-s) \intd s \\
        \end{split}
    \end{equation*}
we can conclude that $\widetilde{x} = x''$ is indeed the corresponding state of the solution, and as we can write
\begin{equation*}
        \begin{split}
            y''(t) &= \derOrd{}{t}{2} \CandD \begin{bmatrix}
                x(t) \\ u(t) 
            \end{bmatrix} = \derOrd{}{t}{2} \CandD \int_0^t (t - s) \begin{bmatrix}
                x''(s) \\ u''(s)
            \end{bmatrix} \intd s,
        \end{split}
    \end{equation*}
we find that $y''$ is the distributional and thus also the classical output $\widetilde{y} = y''$. 
 \end{proof}
Having established these two lemmas allows us to derive two duality results, going from LILO to BIBO stability and vice versa.
\begin{theorem}
\label{thm:CinfLILOImplDualLInfBIBO}
    Let $\systemNode{A}{B}{C}{\mathbf{G}}$ be a $C^\infty$-LILO stable system node such that $A^*$ generates a $C_0$-semigroup. Then the dual system node $\systemNodeStD$ is $C^\infty$-BIBO stable.

    If in addition the space $U$ has the Radon-Nikod\'{y}m property, then $\systemNodeStD$ is $L^\infty$-BIBO stable.
\end{theorem}
\begin{proof}
    Let $\Sigma := \systemNodeSt$ be as in the Proposition and $\Sigma^d := \systemNodeStD$ the dual system node.
    Consider now a generalised solution $(u^d, x^d, y^d)$ of $\Sigma^d$ with $x^d(0) = 0$ and $u^d \in \lpSpacelocDT{\infty}{\posRealNum}{Y^*}$. 
    
    Let $u \in C^\infty_c([0,h],U)$ and denote by $(u,x,y)$ the classical solution of $\Sigma$ with $x(0) = 0$. Then defining $\left[ \begin{smallmatrix}
        K^d(t) \\ L^d(t)
    \end{smallmatrix} \right] = \int_0^t (t - s) \left[ \begin{smallmatrix}
        x^d(t) \\ u^d(t)
    \end{smallmatrix} \right] \intd s$ and noting that by \cite[Lem.~4.7.9]{b_Staffans2005} we know that $(L^d, K^d, \CandD \left[ \begin{smallmatrix}
        K^d(t) \\ L^d(t)
    \end{smallmatrix} \right])$ is a classical solution of $\Sigma^d$, we have that
    \begin{align*}
        &\left| \int_0^h \innerProd{u''(h - t)}{\CandD \begin{bmatrix}
            K^d(t) \\ L^d(t)
        \end{bmatrix} }_{U,U^*} \intd t \right| = \left| \int_0^h \innerProd{y''(h - t)}{L^d(t)}_{Y,Y^*} \intd t \right| \\ &\qquad= \left| \int_0^h \innerProd{y(h - t)}{u^d(t)}_{Y,Y^*} \intd t \right| \leq \| y \|_{\lpSpaceDT{1}{[0,h]}{Y}} \| u^d \|_{\lpSpaceDT{\infty}{[0,h]}{Y^*}} \\ &\qquad\leq c \, \| u \|_{\lpSpaceDT{1}{[0,h]}{U}} \| u^d \|_{\lpSpaceDT{\infty}{[0,h]}{Y^*}},
    \end{align*}
    by Lemmas~\ref{lem:classicalSolDualityIntegrals}~and~\ref{lem:smoothSolDerivativeSol} and the assumed LILO stability with constant $c > 0$.

    Now, $\int_0^h \innerProd{u''(h - t)}{\CandD \begin{bmatrix}
            K^d(t) \\ L^d(t)
        \end{bmatrix} }_{U,U^*} \intd t$ is just the action $\innerProd{y^d}{u(h - \cdot)}$ of the distribution $y^d$ on the test function $u(h - \cdot)$. Hence by the above estimate we find that
        \begin{align*}
           \sup_{\substack{u \in C^\infty_c([0,h],U) \\ \| u \|_{\lpSpaceDT{1}{[0,h]}{U}} = 1}} \left| \innerProd{y^d}{u} \right| \leq  c \, \| u^d \|_{\lpSpaceDT{\infty}{[0,h]}{Y^*}},
        \end{align*}
        and thus by density of $C^\infty_c([0,h],U)$ in $\lpSpaceDT{1}{[0,h]}{U}$ also that 
    \begin{align}
    \label{eq:normingEstimateBIBO}
           \sup_{\substack{u \in \lpSpaceDT{1}{[0,h]}{U} \\ \| u \|_{\lpSpaceDT{1}{[0,h]}{U}} = 1}} \left| \innerProd{y^d}{u} \right| \leq  c \, \| u^d \|_{\lpSpaceDT{\infty}{[0,h]}{Y^*}}.
        \end{align}
    Now for all $u^d \in \CinfComp{\posRealNum}{Y^*}$ we know that $y^d \in C( \posRealNum, U^* )$ and thus in particular $y^d \in \lpSpaceDT{\infty}{[0,h]}{U^*}$. Then Equation~\eqref{eq:normingEstimateBIBO} implies by \cite[Lem.~6.8]{a_WeissObservationOperators1989} that $\| y^d \| \leq  c \, \| u^d \|_{\lpSpaceDT{\infty}{[0,h]}{Y^*}}$ rendering $\Sigma^d$ $C^\infty$-BIBO stable.

    If $U$ further satisfies the Radon-Nikod\'{y}m property then for any $u^d \in \lpSpacelocDT{\infty}{\posRealNum}{Y^*}$ Equation~\eqref{eq:normingEstimateBIBO} already implies that
    $y^d \in \left( \lpSpaceDT{1}{[0,h]}{U} \right)^* = \lpSpaceDT{\infty}{[0,h]}{U^*}$ and $\| y^d \| \leq  c \, \| u^d \|_{\lpSpaceDT{\infty}{[0,h]}{Y^*}}$ \cite[Thm.~1.3.10]{b_AnalysisInBanachSpaces1}. Hence $\Sigma^d$ is $L^\infty$-BIBO stable.    
    \end{proof}

\begin{theorem}
    \label{thm:LinfBIBODualCInfLILO}
    Let $\systemNode{A}{B}{C}{\mathbf{G}}$ be a $L^\infty$-BIBO stable system node such that $A^*$ generates a $C_0$-semigroup. Then the dual system node $\systemNodeStD$ is $L^1$-LILO stable.
\end{theorem}
\begin{proof}
    Let $\Sigma := \systemNodeSt$ be as in the Proposition and $\Sigma^d := \systemNodeStD$ the dual system node.
    Consider then a classical solution $(u^d, x^d, y^d)$ of $\Sigma^d$ with $x^d(0) = 0$ and $u^d \in C^\infty_c(\posRealNum, Y^*)$. Then by \cite[Lem.~6.8]{a_WeissObservationOperators1989} we have that
    \begin{align*}
        \| y^d &\|_{\lpSpaceDT{1}{[0,h]}{U^*}} = 
         \sup_{\substack{u \in \lpSpaceDT{\infty}{[0,h]}{U} \\ \| u \|_{\lpSpaceDT{\infty}{[0,h]}{U}} = 1}} \left| \int_0^h \innerProd{u(t)}{y^d(h - t)}_{U,U^*} \intd t \right| 
    \end{align*}
    Considering then the solution $(u,x,y)$ of $\Sigma$ with and defining $\left[ \begin{smallmatrix}
        K(t) \\ L(t)
    \end{smallmatrix} \right] = \int_0^t (t - s) \left[ \begin{smallmatrix}
        x(t) \\ u(t)
    \end{smallmatrix} \right] \intd s$ we have by \cite[Lem.~4.7.9]{b_Staffans2005} that $\left( L, K, \CandD \left[ \begin{smallmatrix}
        K \\ L
    \end{smallmatrix} \right] \right)$ is a classical solution of $\Sigma$ and by Lemma~\ref{lem:smoothSolDerivativeSol} that $\left( (u^d)'', (x^d)'', (y^d)'' \right)$ is a classical solution of $\Sigma^d$. 
    Rewriting $u(t) = L''(t)$ we can apply Lemma~\ref{lem:classicalSolDualityIntegrals} and the $L^\infty$-BIBO stability with constant $c>0$ to find
    \begin{align*}
        \| y^d \|_{\lpSpaceDT{1}{[0,h]}{U^*}}
        &= \sup_{\substack{u \in \lpSpaceDT{\infty}{[0,h]}{U} \\ \| u \|_{\lpSpaceDT{\infty}{[0,h]}{U}} = 1}} \left| \int_0^h \innerProd{L(t)}{(y^d)''(h - t)}_{U,U^*} \intd t \right| \\
        &= \sup_{\substack{u \in \lpSpaceDT{\infty}{[0,h]}{U} \\ \| u \|_{\lpSpaceDT{\infty}{[0,h]}{U}} = 1}} \left| \int_0^h \innerProd{C\&D \begin{bmatrix}
            K(t) \\ L(t)
        \end{bmatrix}}{(u^d)''(h - t)}_{Y,Y^*} \intd t \right| \\
        &= \sup_{\substack{u \in \lpSpaceDT{\infty}{[0,h]}{U} \\ \| u \|_{\lpSpaceDT{\infty}{[0,h]}{U}} = 1}} \left| \int_0^h \innerProd{ y(t)}{u^d(h - t)}_{Y,Y^*} \intd t \right| \\
        &\leq \sup_{\substack{u \in \lpSpaceDT{\infty}{\posRealNum}{U} \\ \| u \|_{\lpSpaceDT{\infty}{[0,h]}{U}} = 1}} \| y \|_{\lpSpaceDT{\infty}{[0,h]}{Y}} \| u^d \|_{\lpSpaceDT{1}{[0,h]}{Y^*}} \leq c \, \| u^d \|_{\lpSpaceDT{1}{[0,h]}{Y^*}}.
    \end{align*}
    Hence $\Sigma^d$ is $C^\infty$-LILO stable and by Proposition~\ref{prop:CLILOLLILOequiv} also $L^1$-LILO stable. 
 \end{proof}

\section{Outlook}
The results from this note show that the result from finite-dimensional theory --- that BIBO stability of a system implies BIBO stability of the dual system --- does not hold in general for system nodes built on infinite-dimensional spaces. 
Instead it demonstrates that the analogously defined notions of $C^\infty$- and $L^1$-LILO stability provide dual notions to BIBO stability. 

In the SISO case, i.e.\ when both the input and the output space are finite-dimensional, it turns out that all four studied notions --- $C^\infty$-BIBO, $L^\infty$-BIBO, $C^\infty$-LILO and $L^1$-LILO stability --- are characterised by the transfer function being the Laplace transform of a measure of bounded total variation and are thus equivalent. Moreover, in this situation, they are also equivalent to the corresponding properties of the dual system should it exist. Hence for these particular systems the finite-dimensional observation does carry over.  

For the general case the relation between the different BIBO and LILO stability notions --- both of the original system and under duality transformations --- is more complicated. Figure~\ref{fig:BIBOLILIOrelations} summarises the results as presented in this contribution. Note in particular that --- in contrast to the LILO stability case --- it remains an open question so far whether or not $C^\infty-$, $C$- and $L^\infty$-BIBO stability are equivalent as well.
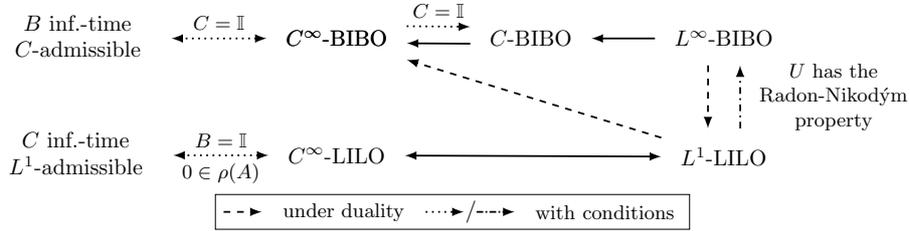
\begin{figure}[h]
\centering
\resizebox{\textwidth}{!}{
\begin{tikzpicture}[stdnode/.style={inner xsep = 3mm,inner ysep=3mm},boxednode/.style={rectangle,thick,draw,inner xsep = 3mm,inner ysep=3mm},roundnode/.style={circle,thick,draw,inner xsep = 0.3mm,inner ysep=0.3mm}]
\node[stdnode] (CinfBIBO) {$C^\infty$-BIBO};
\node[stdnode] (CBIBO) [right= of CinfBIBO] {$C$-BIBO};
\node[stdnode] (LinfBIBO) [right= of CBIBO] {$L^\infty$-BIBO};
\node[stdnode] (CinfLILO) [below= of CinfBIBO] {$C^\infty$-LILO};
\node[stdnode] (LoneLILO) [below= of LinfBIBO] {$L^1$-LILO};
\node[stdnode] (CinfBIBO) {$C^\infty$-BIBO};
\node[stdnode] (BinfTimeCadm) [left=1.5cm of CinfBIBO] {\begin{tabular}{c} $B$ inf.-time \\ $C$-admissible \end{tabular}};
\node[stdnode] (CinfTimeLoneAdm) [left=1.5cm of CinfLILO] {\begin{tabular}{c} $C$ inf.-time \\ $L^1$-admissible \end{tabular}};
\draw[-latex,thick] (LinfBIBO) to (CBIBO);
\draw[-latex,thick] (CBIBO.185) to (CinfBIBO.356);
\draw[-latex,thick,dashed] (LinfBIBO.240) to (LoneLILO.119);
\draw[-latex,thick,dash dot] (LoneLILO.60) to node[right]{\small \begin{tabular}{c} $U$ has the \\ Radon-Nikod\'{y}m \\  property \end{tabular}}(LinfBIBO.301);
\draw[-latex,thick,dashed] (LoneLILO) to (CinfBIBO);
\draw[latex-latex,thick,dotted] (CinfBIBO) -- (BinfTimeCadm) node[midway,above] {\small $C = \mathbb{I}$};
\draw[-latex,thick,dotted] (CinfBIBO.10) -- (CBIBO.169) node[midway,above] {\small $C = \mathbb{I}$};
\draw[latex-latex,thick,dotted] (CinfLILO) -- (CinfTimeLoneAdm) node[midway,above] {\small $B = \mathbb{I}$} node[midway,below] {\small $0 \in \rho(A)$};
\draw[latex-latex,thick] (CinfLILO) to (LoneLILO) ;
\matrix [draw,thin,above right,cells={nodes={anchor=west}}, inner ysep=0.5mm, column sep=0.2cm,anchor=center,ampersand replacement=\&] at ([yshift=-5pt] current bounding box.south) {
    \draw[-latex,thick,dashed](0,0) -- ++ (0.6,0); \& \node{\small under duality}; \&
    \draw[-latex,thick,dotted](0,0) -- ++ (0.6,0); \node at (0.5,0) {/} ; \draw[-latex,thick,densely dash dot](0.8,0) -- ++ (0.6,0) ; \& \node{\small with conditions};\\
  };
\end{tikzpicture}
}
\caption{Relationship between different BIBO and LILO stability notions 
}
\label{fig:BIBOLILIOrelations}
\end{figure}

From a practical point of view, this equivalence of BIBO and LILO stability in the SISO case does --- on a first glance --- not seem to allow significant simplifications of questions related to BIBO stabiliy, as both properties are reduced to the same characterisation via the inverse Laplace transform.

However, working with $L^1$-spaces instead of $L^\infty$ can anyhow be advantageous. Even though $L^1$ is still also only a non-reflexive Banach space --- with all the problems arising therefrom --- other of its properties make it a more suitable space to work with in systems theory compared to $L^\infty$. An example for this is the strong continuity of the various shift semigroups, which holds on $L^1$ but fails in general on $L^\infty$, and which appears e.g.\ in the study of well-posedness of system nodes or realisation theory.

The results presented here also provide further indication that the case of infinite-dimensional input and/or output spaces indeed seems to be considerably more involved and complex compared to the SISO case. The observation that in the latter BIBO and LILO stability are equivalent follows directly from the existence of a measure of bounded total variation encoding the input-output behaviour via a convolution operator. But as Remark~\ref{rem:BIBOLILOnonEquiv} shows that this result cannot hold in general, also a corresponding input-output representation using operator-value measures seems to be excluded in the non-SISO case.


\end{document}